\documentclass[oneside,english]{article}
\usepackage[T1]{fontenc}
\usepackage{amsthm}
\usepackage{amssymb}
\usepackage[all]{xy}
\usepackage{amsmath}
\usepackage{comment}
\usepackage{color}
\usepackage[nocompress]{cite}
\usepackage{tikz-cd}
\usepackage{mathtools}
\usepackage{covington}

\makeatletter
\newcommand{\nc}{\newcommand}
\newtheorem{thm}{Theorem}
\theoremstyle{plain}
\nc{\bthm}{\begin{thm}} \nc{\ethm}{\end{thm}}
\newtheorem{prop}[thm]{Proposition}
\nc{\bprp}{\begin{prop}} \nc{\eprp}{\end{prop}}
\newtheorem{fact}[thm]{Fact}
\nc{\bfct}{\begin{fact}} \nc{\efct}{\end{fact}}
\newtheorem{prob}[thm]{Problem}
\nc{\bprb}{\begin{prob}} \nc{\eprb}{\end{prob}}
\newtheorem{lem}[thm]{Lemma}
\nc{\blem}{\begin{lem}} \nc{\elem}{\end{lem}}
\newtheorem{claim}[thm]{Claim}
\nc{\bclm}{\begin{claim}} \nc{\eclm}{\end{claim}}
\newtheorem{cor}[thm]{Corollary}
\nc{\bcor}{\begin{cor}} \nc{\ecor}{\end{cor}}
\newtheorem{conj}[thm]{Conjecture}
\nc{\bcnj}{\begin{conj}} \nc{\ecnj}{\end{conj}}
\theoremstyle{definition}
\newtheorem{defn}[thm]{Definition}
\nc{\bdfn}{\begin{defn}} \nc{\edfn}{\end{defn}}
\newtheorem{observation}[thm]{Observation}
\nc{\bobs}{\begin{observation}} \nc{\eobs}{\end{observation}}
\theoremstyle{remark}
\newtheorem{rem}[thm]{Remark}
\nc{\brem}{\begin{rem}} \nc{\erem}{\end{rem}}
\newtheorem{cnv}[thm]{Convention}
\nc{\bcnv}{\begin{cnv}} \nc{\ecnv}{\end{cnv}}
\newtheorem{exam}[thm]{Example}
\nc{\bexm}{\begin{exam}} \nc{\eexm}{\end{exam}}

\newtheorem{question}[thm]{Question}
\newtheorem{innercustomthm}{Question}
\newenvironment{myquestion}[1]
{\renewcommand\theinnercustomthm{#1}\innercustomthm}
{\endinnercustomthm}
\nc{\bpf}{\begin{proof}} \nc{\epf}{\end{proof}}
\nc{\be}{\begin{enumerate}}
	\nc{\ee}{\end{enumerate}}
\nc{\bi}{\begin{itemize}}
	\nc{\itm}{\item}
	\nc{\ei}{\end{itemize}}
\nc{\invlim}{\lim_{\leftarrow}}
\nc{\dirlim}{\lim_{\rightarrow}}
\nc{\mm}{\mathbf{m}}
\nc{\nn}{\mathbf{n}}
\nc{\FF}{\mathcal{F}}
\nc{\CC}{\mathcal{C}}
\nc{\Span}{\operatorname{span}}
\nc{\Img}{\operatorname{Im}}
\nc{\rank}{\operatorname{rank}}
\nc{\proj}{\operatorname{proj}}
\nc{\F}{\mathbb{F}}
\nc{\Z}{\mathbb{Z}}
\nc{\Q}{\mathbb{Q}}
\nc{\Br}{\operatorname{Br}}

\title{Jordan-Holder Theorem for profinite groups and applications}

\author{Tamar Bar-On and Nikolay Nikolov}
\date{\today}
\begin{document}
	\maketitle
	\begin{abstract}
		We generalize the notions of composition series and composition factors for profinite groups, and prove a profinite version of the Jordan-Holder Theorem. We apply this to prove a Galois Theorem for infinite prosolvable extensions. In addition, we investigate the connection between the abstract and topological composition factors of a nonstrongly complete profinite group.
	\end{abstract}
	\section{Composition factors and Jordan-Holder Theorem}
	The composition factors of a finite group are one of its most important invariants, and help translating structural questions in finite group theory to questions on the structure of finite simple groups - which have been classified. 
	
	The fundamental fact about composition factors and their multiplicity is that they do not depend on the composition series:
	\begin{thm}[Jordan-Holder Theorem for finite groups]
		Let $G$ be a finite group and let $\{e\}=G_n \unlhd \cdots G_1\unlhd G_0=G$ and   $\{e\}=H_m \unlhd \cdots H_1\unlhd H_0=G$ be two composition series for $G$. Then for every finite simple group $S$, $|\{0\leq i\leq n-1: G_i/G_{i+1}\cong S\}|=|\{0\leq i\leq m-1: H_i/H_{i+1}\cong S\}|$. In particular $m=n$ and the sets $\{G_i/G_{i+1} | \ 0 \leq i \leq n-1\}$ and $\{H_i/H_{i+1} \ | \ 0 \leq i \leq m-1\}$ are equal.
	\end{thm} 
	For a proof see for example \cite{Baumslag2006simple}.
	For infinite groups the theorem no longer holds, as can be seen for example, by taking $G=\Z$ and considering the two composition series: $\Z\unrhd p\Z\unrhd p^2\Z\unrhd...$ and  $\Z\unrhd q\Z\unrhd q^2\Z\unrhd...$ for $p\ne q$ primes.
	
	Many generalizations to Jordan-Holder Theorem have been proven, as  in \cite{Birkhoff1934Transfinite} for well-ordered ascending $T$-invariant series, and the Jordan-Holder Theorem for composition factors of modules (see, for example, \cite[Theorem 3.11]{Erdmann2018}). 
	
	The purpose of this paper is to establish a generalization of the Jordan-Holder Theorem for profinite groups. More precisely, we define composition factors and composition series for profinite groups, and show that the composition factors of a profinite groups, as well as their multiplicity, are well-defined and independent on the choice of a composition series.  
	
	Through this paper, unless stated otherwise (which will be relevant only in Section 3) subgroups in a profinite group are assumed to be closed and quotients are assumed to be topological quotients.
	
	Our first observation is that for profinite groups, subnormal series are always replaced by \textit{accessible series}:
	\begin{defn}\cite[Chapter 8.3]{ribes2000profinite}
		Let $G$ be a profinite group and $H\leq G$. An accessible series from $G$ to $H$ over some ordinal $\mu$ is a series $G=G_0>...>G_{\lambda}>...>G_{\mu}=H$ which satisfies the following property:
		\begin{enumerate}
			\item For every $\alpha<\mu$, $G_{\alpha+1}\unlhd G_{\alpha}$.
			\item For every \textit{limit} ordinal $\alpha<\mu$, $G_{\alpha}=\bigcap_{\beta<\alpha}G_{\beta}$.
		\end{enumerate}
		If such a series exists we say that $H$ is an \textit{accessible subgroup}. In case $H=\{e\}$ we refer to an accessible series from $G$ to $\{e\}$ simply as \textit{an accessible series for $G$}. 
	\end{defn}
	Now we can define a \textit{composition series}:
	\begin{defn}
		Let $G$ be a profinite group. A \textit{composition series} for $G$ is an accessible series $G=G_0>...>G_{\lambda}>...>G_{\mu}=\{e\}$ such that $G_{\lambda}/G_{\lambda+1}$ is a finite simple nontrivial group for every $\lambda<\mu$. The quotients $G_{\lambda}/G_{\lambda+1}$ for all $\lambda<\mu$ are called the \textit{composition factors} of $G$.
	\end{defn}
	\begin{prop}\label{well defined composition factors}
		Every profinite group admits a composition series, and the composition factors are well defined. Moreover, the set of composition factors of $G$ is precisely the set of all composition factors of all finite topological quotients of $G$.
	\end{prop}
	\begin{proof}
		First we show that every profinite group admits a composition series. Let $\eta=|G|$ and $\nu=2^{\eta}$. We define a composition series over $\nu$ recursively, as follows. Assume that for all $\alpha<\beta$ $G_{\alpha}$ has been defined. If $\beta=\gamma+1$ for some $\gamma$ then we define $G_{\beta}$ to be a maximal open normal subgroup of $G_{\gamma}$ in case $G_{\gamma}\ne \{e\}$ and to be $\{e\}$ otherwise. If $\beta$ is a limit ordinal then we define $G_{\beta}=\bigcap_{\alpha<\beta}G_{\alpha}$. We claim that $G_{\nu}=\{e\}$. Indeed, otherwise we can use the axiom of choice to define a one-to-one map $\nu\to G$ by sending each $\alpha$ to an element in $G_{\alpha}\setminus G_{\alpha+1}$. Let $\mu$ be the first ordinal for which $G_{\mu}=\{e\}$, then $G=G_0>...>G_{\lambda}>...>G_{\mu}=\{e\}$ is a composition series for $G$. 
		
		Now we prove that for every composition series $\{G_{\alpha}\}_{\alpha<\mu}$, the set \[C_G:=\{G_{\alpha}/G_{\alpha+1}\}_{\alpha<\mu} \] equals the set of composition factors of all finite quotients of $G$, and that will imply that $C_G$ does not depend on the choice of a composition series.
		
		First we show the following: let $\{G_{\alpha}\}_{\alpha<\mu}$ be a composition series for $G$ and $A$ a finite image of $G$ with some canonical projection $\varphi_A:G\to A$. Then the image series   $\{\varphi_A({G_{\alpha}})\}_{\alpha<\mu}$ is clearly a subnormal series in $A$ which ends in $\{e\}$. Denote the different subgroups in this series by $A=A_0\unrhd A_1\unrhd \cdots A_n=\{e\}$. For $0<i\leq n$, let $\alpha$ be the first ordinal such that $\varphi_A({G_{\alpha}})=A_i$. We claim that $\alpha$ is a successor ordinal. Otherwise, for every $\beta<\alpha$ $\varphi_A({G_{\beta}})\geq A_{i-1}$. Then by \cite[Proposition 2.1.4]{ribes2000profinite} $$\varphi_A({G_{\alpha}})=\varphi_A({\bigcap_{\beta<\alpha}G_{\beta}})=\bigcap_{\beta<\alpha}\varphi_A({G_{\alpha}})\geq A_{i-1}$$
		Now let $\alpha=\beta+1$. Then $\varphi_A(G_{\beta})=A_{i-1}$ and since $G_{\beta}/G_{\beta+1}$ is simple, $G_{\beta}/G_{\beta+1}\cong A_{i-1}/A_i$. In particular, $A=A_0\unrhd A_1\unrhd \cdots A_n=\{e\}$ is a composition series and all its factors belongs to the set of the factors $\{G_{\alpha}/G_{\alpha+1}:\alpha<\mu\}$. We conclude that the set of the factors $\{G_{\alpha}/G_{\alpha+1}:\alpha<\mu\}$ contains the set of composition factors of all finite quotients of $G$.
		
		Conversely, let $\alpha<\mu$. There is some finite quotient $A$ of $G$ in which $\varphi_A({G_{\alpha}})\ne\varphi_A({G_{\alpha+1}})$. By the argument above, the different members of $\{\varphi_A(G_{\alpha})\}_{\alpha<\mu}$ form a composition series $A=A_0\unrhd A_1\unrhd \cdots A_n=\{e\}$ of $A$. Hence if $\varphi_A(G_{\alpha})=A_i$ then $\varphi(G_{\alpha+1})=A_{i+1}$ and $G_{\alpha}/G_{\alpha+1}\cong A_{i}/A_{i+1}$, again from the simplicity of $G_{\alpha}/G_{\alpha+1}$. In particular, $G_{\alpha}/G_{\alpha+1}$ is a composition factor of $A$, and we are done.
	\end{proof}
	Recall that a profinite group is called prosolvable if all its finite quotients are solvable.
	\begin{cor}\label{equibalent criterion for prosolvable}
		Let $G$ be a profinite group. Then $G$ is prosolvable if and only if all its composition factors are abelian.
	\end{cor}
	Now we wish to prove the second part of the main goal of this paper:
	\begin{thm}\label{main Theorem}
		Let $G$ be a profinite group and $S$ a finite simple group. Let $G=G_0>...>G_{\alpha}>...>G_{\mu}=\{e\}$ and $G=H_0>...>H_{\lambda}>...>H_{\Delta}=\{e\}$  be two composition series for $G$. Let $S$ be some finite simple group. Denote by $\mathcal{S}_1=\{\alpha<\mu:G_{\alpha}/G_{\alpha+1}\cong S\} $ and $\mathcal{S}_2=\{\lambda<\Delta:H_{\lambda}/H_{\lambda+1}\cong S\}$. Then $|\mathcal{S}_1|=|\mathcal{S}_2|$.
	\end{thm}
	
	First we prove a stronger result in the case of $\mathcal{S}_1$ being finite. Let $S$ be a simple group. For a finite quotient $A$ of $G$ we denote by $n_{S,A}$ the number of appearances of $S$ as a factor in some, and hence any, composition series of $A$. 
	\begin{lem}\label{finite case}
		Let $n$ be some natural number. Then $|\mathcal{S}_1|=n$ if and only if $$n=\max\{n_{S,A} \ | \ A \text { is a finite quotient of $G$} \}$$
	\end{lem}
	\begin{proof}
		As was shown in the proof of Proposition \ref{well defined composition factors}, the different image subgroups $\varphi_A({G_{\alpha}}), \alpha< \mu$ induce a composition series $A=A_0\unrhd \cdots A_m=\{e\}$ of $A$, for every finite image $A$ of $G$. Moreover, there are $\alpha_0<...<\alpha_m<\mu$ such that $G_{\alpha_i}/G_{\alpha_i+1}\cong A_i/A_{i+1}$. Hence, $|\mathcal{S}_1|\geq n_{S,A}$ for every finite quotient $A$ of $G$.
		
		On the other hand, assume that $l$ is a natural number and there are  $\alpha_0<...<\alpha_l<\mu$ such that $G_{\alpha_i}/H_{\alpha_i+1}\cong S$ for all $0\leq i\leq l$. There is some finite quotient $A$ of $G$ for which $\varphi_A({G_{\alpha_i}})\ne \varphi_A({G_{\alpha_i+1}})$ for all $0\leq i\leq l$. In particular, if   we denote by $A=A_0\unrhd \cdots A_m=\{e\}$ the induced composition series of $A$, and $\varphi_A({G_{\alpha_i}})=A_{k_i}$ for some $0\leq k_i<m$ then $\varphi_A({G_{\alpha_i+1}})=A_{k_i+1}$. Moreover, $\varphi_A({G_{\alpha_i}})\ne \varphi_A({G_{\alpha_j}})$ for all $j>i$. Thus, there are $0\leq k_1<...<k_l\leq m$ such that $A_{k_i}/A_{k_i+1}\cong \varphi_A({G_{\alpha_i}})/\varphi_A({G_{\alpha_i+1}})\cong S$ and we conclude that $\max\{n_{S,A}:A\text { is a finite quotient of $G$}\}\geq n_{S,A}\geq |\mathcal{S}_1|$.
	\end{proof}
	Before we can prove Theorem \ref{main Theorem} in full generality we need some Lemmas.
	\begin{lem}\label{case distinction}
		Let $A\ne B$ and $H$ be subgroups of a given group $G$ such that $B$ is strictly contained in $A$ and $A\cap H=B\cap H$. Then $AH\ne BH$. 
	\end{lem}
	\begin{proof}
		Assume by contradiction that $AH=BH$. Let $a\in A\setminus B$, then $a\in BH$. Thus there exist $b\in B,h\in H$ such that $a=bh$. Since $B\leq A$, $b\in A\Rightarrow h\in A$. Thus $h\in A\cap H=B\cap H\Rightarrow h\in B$ and we get that $a\in B$. A contradiction.
	\end{proof}
	\begin{lem}\label{different in pieces}
		Let $G$ be a profinite group and $A,B\leq G$ such that $B$ is strictly contained in $A$. Let $\{H_{\alpha}:\alpha<\mu\}$ be an accessible series from $G$ to some accessible subgroup $H$ and assume that $AH\ne BH$. Then there exists some $\alpha$ such that $(A\cap H_{\alpha})H_{\alpha+1}\ne (B\cap H_{\alpha})H_{\alpha+1}$.
	\end{lem}
	\begin{proof}
		Let $\beta<\mu$ be the first ordinal such that $AH_{\beta}\ne BH_{\beta}$. We claim that $\beta$ is a successor ordinal. Otherwise, $H_{\beta}=\bigcap_{\alpha<\beta} H_{\alpha}$ and by \cite[Proposition 2.1.4]{ribes2000profinite} $AH_{\beta}=A\bigcap_{\alpha<\beta} H_{\alpha}=\bigcap_{\alpha<\beta}AH_{\alpha}= \bigcap_{\alpha<\beta}BH_{\alpha}=B\bigcap_{\alpha<\beta} H_{\alpha}=BH_{\beta}$, a contradiction. Thus $\beta=\alpha+1$ for some $\alpha$. Notice that $(A\cap H_{\alpha})H_{\alpha+1}=AH_{\alpha+1}\cap H_{\alpha}$ and the same holds for $B$. Since $AH_{\alpha+1}\ne BH_{\alpha+1}$ and $AH_{\alpha+1}H_{\alpha}=AH_{\alpha}=BH_{\alpha}=BH_{\alpha+1}H_{\alpha}$, Putting $A'=AH_{\alpha+1}, B'=BH_{\alpha+1}, H=H_{\alpha}$, Lemma \ref{case distinction} applies the required.
	\end{proof}
	
	\begin{lem}\label{intersection of composition series}
		Let $\{G_{\alpha}:\alpha<\mu\}$ be a composition series for $G$, and $H$ be an accessible subgroup of $G$. Then $G_{\alpha}\cap H$ is a composition series for $H$.
	\end{lem}
	\begin{proof}
		It is clear that this series is accessible. Fix an $\alpha<\mu$, we want to show that $(G_{\alpha}\cap H)/(G_{\alpha+1}\cap H)$ is either trivial or isomorphic to $G_{\alpha}/G_{\alpha+1}$.  Let $G=H_0\unrhd...\unrhd H_\beta\unrhd ...\unrhd H_{\Delta}=H$ be an accessible series from $G$ to $H$. The claim is now translated to $(G_{\alpha}\cap H_{\Delta})/(G_{\alpha+1}\cap H_{\Delta})$ is either trivial or isomorphic to $G_{\alpha}/G_{\alpha+1}$. We will prove by induction on $\beta$ that in fact for every $\beta\leq \Delta$ $(G_{\alpha}\cap H_{\beta})/(G_{\alpha+1}\cap H_{\beta})$ is either trivial or isomorphic to $G_{\alpha}/G_{\alpha+1}$. 
		
		Obviously, for $\beta=0$ the claim holds. Now let $\beta\leq \Delta$ and assume the claim holds for all $\gamma<\beta$. If $H_{\gamma}\cap G_{\alpha}=H_{\gamma}\cap G_{\alpha+1}$ for some $\gamma<\beta$ then the same holds for $H_{\beta}$. Now assume that $G_{\alpha}\cap H_{\gamma}/G_{\alpha+1}\cap H_{\gamma}\ne \{e\}$ for every $\gamma<\beta$. 
		
		First case: $\beta$ is limit. By induction assumption, $(G_{\alpha}\cap H_{\gamma})/(G_{\alpha+1}\cap H_{\gamma})\cong G_{\alpha}/G_{\alpha+1}$ for every $\gamma<\beta$. Recall that $G_{\alpha}/G_{\alpha+1}$ is finite. We get that the natural inclusion $(G_{\alpha}\cap H_{\gamma})/(G_{\alpha+1}\cap H_{\gamma})\to G_{\alpha}/G_{\alpha+1}$ is in fact an epimorphism. Hence, $(G_{\alpha}\cap H_{\gamma})G_{\alpha+1}=G_{\alpha}$ for all $\gamma<\beta$. Using \cite[Proposition 2.1.4]{ribes2000profinite} implies the same for $H_{\beta}$ and we conclude that the natural inclusion $(G_{\alpha}\cap H_{\beta})/(G_{\alpha+1}\cap H_{\beta})\to G_{\alpha}/G_{\alpha+1}$ is an isomorphism, as required. 
		
		Second case: $\beta=\gamma+1$. Then $H_{\beta}\unlhd H_{\gamma}$ implies that $(H_{\beta}\cap G_{\alpha})/(H_{\beta}\cap G_{\alpha+1})\unlhd (H_{\gamma}\cap G_{\alpha})/(H_{\gamma}\cap G_{\alpha+1})\cong G_{\alpha}/G_{\alpha+1}$. The last isomorphism is due to the induction assumption. Now the claim follows by $G_{\alpha}/G_{\alpha+1}$ being simple.
	\end{proof}
	
	Now we are ready to prove the main theorem.
		\begin{proof}[Proof of Proposition \ref{main Theorem}]
			Let $\{G_{\alpha}\}_{\alpha<\mu}$ and $\{H_{\beta}\}_{\beta<\nu}$ be two composition series. Let $S$ be a simple group and $\alpha<\mu $ be some ordinal such that $G_{\alpha}/G_{\alpha+1}\cong S$. For more simplicity denote $G_{\alpha}=A, G_{\alpha+1}=B$. By Lemma \ref{different in pieces} there is some $\beta<\nu $ such that $(A\cap H_{\beta})H_{\beta+1}\ne(B\cap H_{\beta})H_{\beta+1}$. By Lemma \ref{intersection of composition series}, $\{G_{\alpha}\cap H_{\beta}:\alpha<\mu\}$ is an composition series of $H_{\beta}$. Hence, as was explained in the proof of \refeq{well defined composition factors}, its image is a composition series for $H_{\beta}/H_{\beta+1}$. In particular, $((A\cap H_{\beta})H_{\beta+1})/H_{\beta+1}$ and $((B\cap H_{\beta})H_{\beta+1})/H_{\beta+1}$ are subnormal subgroups of the simple group $H_{\beta}/H_{\beta+1}$. We conclude that $(A\cap H_{\beta})H_{\beta+1}=H_{\beta}$ and $(B\cap H_{\beta})H_{\beta+1}=H_{\beta+1}$.
			By the proof of Lemma \ref{intersection of composition series}, either $B\cap H_{\beta}=A\cap H_{\beta}$ or $(A\cap H_{\beta})/(B\cap H_{\beta})\cong A/B$. The first option is impossible since  $ (A\cap H_{\beta})H_{\beta+1}\ne  (B\cap H_{\beta})H_{\beta+1}$. Thus $(A\cap H_{\beta})/(B\cap H_{\beta})\cong A/B$. Now $H_{\beta}/H_{\beta+1}=((A\cap H_{\beta})H_{\beta+1})/(B\cap H_{\beta})H_{\beta+1}$. Look at the natural epimorphism $(A\cap H_{\beta})/(B\cap H_{\beta})\to ((A\cap H_{\beta})H_{\beta+1})/(B\cap H_{\beta})H_{\beta+1}=H_{\beta}/H_{\beta+1}$. Since $(A\cap H_{\beta})/(B\cap H_{\beta})$ is simple, this is in fact an isomorphism. We conclude that $H_{\beta}/H_{\beta+1}\cong S$.

			Now let $\mathcal{S}_1$ be the set of all $\alpha<\mu$ such that $G_{\alpha}/G_{\alpha+1}\cong S$ and $\mathcal{S}_2$ be the set of all $\beta<\nu$ such that $H_{\beta}/H_{\beta+1}\cong S$. We define a map $S_1\to S_2$ by choosing for every $\alpha\in \mathcal{S}_1$ the first $\beta\in \mathcal{S}_2$ such that $(G_{\alpha}\cap H_{\beta})H_{\beta+1}\ne  (G_{\alpha+1}\cap H_{\beta})H_{\beta+1}$.

			By Lemma \ref{different in pieces} such $\beta$ always exists. By the above computations, such $\beta$ satisfies $H_{\beta}/H_{\beta+1}\cong S$ and thus $\beta\in \mathcal{S}_2$. This map is one-to-one. Indeed, assume $\alpha<\gamma\in \mathcal{S}_1$ and $(G_{\alpha}\cap H_{\beta})H_{\beta+1}\ne  (G_{\alpha+1}\cap H_{\beta})H_{\beta+1}$.
			
			In particular, as we computed above, $(G_{\alpha+1}\cap H_{\beta})H_{\beta+1}=H_{\beta+1}$. Since $\gamma\geq \alpha+1$, $$H_{\beta+1}\leq (G_{\gamma}\cap H_{\beta})H_{\beta+1}\leq (G_{\alpha+1}\cap H_{\beta})H_{\beta+1}=H_{\beta+1}$$ The same goes for $G_{\gamma+1}$, so $(G_{\gamma}\cap H_{\beta})H_{\beta+1}=(G_{\gamma+1}\cap H_{\beta})H_{\beta+1}$. In conclusion, $|\mathcal{S}_1|\leq |\mathcal{S}_2|$. By symmetry we conclude that $|\mathcal{S}_2|\leq |\mathcal{S}_1|$ and thus we have an equality. 
			
		\end{proof}
		We end this section by computing the length of a composition series. First we need to define the following invariant of closed subgroups.
		\begin{defn}\label{definition of local weight}
			Let $G$ be a profinite group and $H$ a closed non open subgroup of $G$. We define $\omega_0(G/H)$ to be the cardinality of the set of all open subgroups of $G$ containing $H$. If $H=\{e\}$ then we simply write $\omega_0(G)$. The cardinal $w_0(G/H)$ is in fact equal to the cardinality of any set of open subgroups of $G$ whose intersection is $H$ and which is filtered from below. Indeed, assume $\mathcal{A}$ is such a set. Let $U$ be an open subgroup containing $H$. Since $\bigcap\mathcal{A}=H\subseteq U$ there are some open subgroups $O_1,...,O_n$ in $\mathcal{A}$ such that  $O_1\cap...\cap O_n\subseteq U$. By the filtration property, there is some $O\in \mathcal{A}$ such that $O\leq U$. Hence using the axiom of choice we can define a map from the set of all open subgroups containing $H$ to $\mathcal{A}$ sending an open subgroup $U$ to a open subgroup  $O\in\mathcal{A}$ such that $O\leq U$. Since $O$ is open, this map is finite to 1, so the cardinality of the set of all open subgroups of $G$ containing $H$ is less or equal then the cardinality of $\mathcal{A}$. The converse clearly holds, thus we have an equality. 
		\end{defn}
		\begin{rem}
			The definition of $\omega_0(G/H)$ is different then the definition that was given in \cite[Section 2.6]{ribes2000profinite}. However, one can show that these two definitions are equivalent. 
		\end{rem}
		\begin{lem}\label{lengh of accessible}
			Let $H$ be an infinite-indexed accessible subgroup of $G$. Assume that $G=H_0>...>H_{\lambda}>...>H_{\mu}=H$ is an accessible series from $G$ to $H$ such that all quotients $H_{\lambda}/H_{\lambda+1}$ are finite nontrivial. Then $|\mu|=\omega_0(G/H)$.
		\end{lem}
		\begin{proof}
			Clearly $G/H\cong {\invlim}_{H\leq U\leq_oG} G/U$. Here the inverse limit considered as an inverse limit of spaces, as the subgroups are not assumed to be normal. For every $\lambda$ we choose some $U$ such that $\varphi_{G/U}({H_{\lambda}})\ne \varphi_{G/U}({H_{\lambda+1}})$ where $\varphi_{G/U}$ denoted the natural quotient map $G\to G/U$. Since the images $\varphi_{G/U}({H_{\lambda}})$ form a chain of subsets in each $G/U$, which is a finite space, the map adjoining to each $\lambda<\mu$ such $U$ is finite to 1. Hence, $|\mu|\leq \omega_0(G/H)$. 
			
			On the other hand, we will prove by transfinite induction that $\omega_0(G/H_{\lambda})\leq |\lambda|$. The proof is very similar to the proof of \cite[Theorem 2.6.4 (c)]{ribes2000profinite}. For $\lambda=1$ it is obvious. Now assume that $\lambda=\gamma+1$. Since $H_{\gamma+1}$ is an open subgroup of $H_{\gamma}$ then there exists some $V\leq_oG$ such that $V\cap H_{\gamma}\leq H_{\gamma+1}$. It can be shown that $\{VH_{\gamma+1}\cap U\}_{H_{\gamma}\leq U\leq_oG}$ is a set of filtered from below open subgroups whose intersection equals $H_{\gamma+1}$. Indeed, the fact that it is filtered from below comes from the filtration property of the set ${H_{\gamma}\leq U\leq_oG}$. As for the intersection, $$\bigcap_{H_{\gamma}\leq U\leq_oG}(VH_{\gamma+1})\cap U=\bigcap_{H_{\gamma}\leq U\leq_oG}(V\cap U)H_{\gamma+1}$$ $$=H_{\gamma+1}\bigcap_{H_{\gamma}\leq U\leq_oG}(V\cap U)=H_{\gamma+1}(V\bigcap_{H_{\gamma}\leq U\leq_oG}U)=H_{\gamma+1}(V\cap H_{\gamma})=H_{\gamma+1}$$ The second equality is \cite[Proposition 2.1.4]{ribes2000profinite}.
			
			Thus $\omega_0(G/H_{\gamma+1})\leq\omega_0(G/H_{\gamma})\leq|\gamma|=|\gamma+1|$. 
			
			Now assume that $\lambda$ is a limit cardinal.
			Let $U$ be an open subgroup of $G$ containing $H_\lambda$. We have $\cap _{\gamma <\lambda} (H_\gamma \cap (G \backslash U)) =H_\lambda \backslash U =\emptyset$ and from the compactness of $G \backslash U$ we deduce that $H_\gamma \backslash U= \emptyset$ for some $\gamma < \lambda$, i.e. $H_\gamma \leq U$. This shows \[ \{U:H_{\lambda}\leq U\leq_oG\}=\bigcup_{\gamma<\lambda}\{U:H_{\gamma}\leq U\leq_oG\}, \] hence \[ |\{U:H_{\lambda}\leq U\leq_oG\}|\leq \sum_{\gamma<\lambda}|\{U:H_{\gamma}\leq U\leq_oG\}|\leq\sum_{\gamma<\lambda}|\gamma|\leq |\lambda|.\]
		\end{proof}
		\section{Galois Theorem for infinite extensions}
		Now we present a generalization of Galois Theorem for infinite extensions, using composition series of profinite groups. 
		\begin{defn}
			Let $K/F$ be a separable extension.
			\begin{enumerate}
				\item We say that $K$ is prosolvable if $\operatorname{Gal}(N/F)$ is prosolvable, for $N$ being the Galois closure of $K$ over $F$.
				\item We say that $K$ is \textit{solvable by radicals} if there exists a field $L$ containing $K$, an ordinal $\mu$, and a series of field extensions $F=L_0\leq ...\leq L_{\alpha}<...\leq L_{\mu}=L$ which satisfying the following properties:
				\begin{itemize}
					\item For every $\alpha<\mu$, $L_{\alpha+1}/L_{\alpha}$ is a \textit{radical} extension, i.e, $L_{\alpha+1}=L_{\alpha}[\sqrt[n]{a}]$ for some natural number $n$ and $a\in L_{\alpha}$.
					\item For every $\alpha<\mu$ limit, $L_{\alpha}=\bigcup_{\beta<\alpha}L_{\beta}$.
				\end{itemize} 
			\end{enumerate}
		\end{defn}
		\begin{rem}
			Since a quotient of a prosolvable group is prosolvable, $K/F$ is a prosolvable extension if and only if $K$ is contained in some Galois extension $N/F$ such that $\operatorname{Gal}(N/F)$ is prosolvable.
		\end{rem}
		\begin{thm}[Galois Theorem for infinite extensions]
			Let $K/F$ be a separable extension, then $K$ is prosolvable if and only if it is solvable by radicals.
		\end{thm}
		\begin{proof}
			First assume that $K/F$ is prosolvable. Let $N$ be the Galois closure of $K$. By assumption, $\operatorname{Gal}(N/K)$ is prosolvable. By Corollary \ref{equibalent criterion for prosolvable} $G=\operatorname{Gal}(N/F)$ admits a composition series $G=G_0\geq...G_{\alpha}\geq ...\geq G_{\mu}=\{e\}$ all its factors are cyclic. Let $N_{\alpha}=N^{G_{\alpha}}$ and $L$ be the compositum of $N$ with $F[\mu_{\infty}]$ where $F[\mu_{\infty}]$ denotes the field extension of $F$ obtained by adjoining to $F$ all $n$-th roots of unity for every natural number $n$. Obviously, $K\leq L$ and we have the following chain of Galois extensions: $F=L_0\leq ...\leq L_{n+1}=L_n[\sqrt[n!]{1}]\leq ...\leq L_{\omega}=F[\mu_{\infty}]\leq L_{\omega+1}=L_{\omega}N_{1}\leq ...\leq L_{\omega+\alpha}=L_{\omega}N_{\alpha}\leq ...\leq L_{\omega+\mu}=L_{\omega}N=L$. Obviously, for every $n$, $L_{n+1}/L_n$ is radical. For every $\alpha$, $L_{\omega+\alpha+1}/L_{\omega+\alpha}$ is cyclic over a field containing all roots of unity and thus radical too. 
			
			For the second direction assume that $K\leq N$ and there is a chain of radical extensions $F=N_0\leq ...\leq N_{\alpha}\leq ...N_{\mu}=N$. Let $L=F[\mu_{\infty}]N$. Consider the following chain of field extensions: $F=L_0\leq ...\leq L_{n+1}=L_n[\sqrt[n!]{1}]\leq ...\leq L_{\omega}=F[\mu_{\infty}]\leq L_{\omega+1}=L_{\omega}N_{1}\leq ...\leq L_{\omega+\alpha}=L_{\omega}N_{\alpha}\leq ...\leq L_{\omega+\mu}=L_{\omega}N=L$. We define a new chain of field extensions $\{T_{\alpha}\}_{\alpha<\omega+\mu}$ as follows: Let $T_0=F$. For every $\alpha<\omega+\mu$ limit define $T_{\alpha}=\bigcup_{\beta<\alpha} T_{\beta}$. For every $\alpha<\omega+\mu$, let $a$ be such that $L_{\alpha+1}=L_{\alpha}[\sqrt[n]{a}]$ and define $T_{\alpha+1}=T_{\alpha}[\sqrt[n]{\sigma(a)}:\sigma\in \operatorname{Gal}(T_{\alpha}/F)]$. We claim that $T_{\alpha}$ is a Galois extension containing $L_{\alpha}$. We prove it by transfinite induction. When $\alpha=0$ this is clear. Assume that $T_{\alpha}$ containing $L_{\alpha}$, then clearly $L_{\alpha+1}=L_{\alpha}[\sqrt[n]{a}]\subseteq T_{\alpha}[\sqrt[n]{\sigma(a)}:\sigma\in \operatorname{Gal}(T_{\alpha}/F)]=T_{\alpha+1}$. Moreover, $T_{\alpha+1}$ is Galois over $F$ as the compositum of $T_{\alpha}$ with the splitting field over $F$ of the polynomial $f_a(x^n)$, where $f_a$ is the minimal polynomial of $a$ over $F$. When $\alpha$ is a limit ordinal the claim is immediate. Hence $T_{\omega+\mu}$ is a Galois extension containing $K$. In addition, for every $\alpha<\omega+\mu$, $LT_{\alpha+1}/T_{\alpha}$ is abelian. Let $G_{\alpha}=\operatorname{Stab}_{\operatorname{Gal}(L/F)}(T_{\alpha})$, we get that $G_{\alpha}$ is an accessible series for $\operatorname{Gal}(L/F)$, $G_{\alpha+1}/G_{\alpha}$ is abelian and $\bigcap_{\alpha<\omega+\mu}G_{\alpha}=\{e\}$. We conclude from Corollary \ref{equibalent criterion for prosolvable} that $\operatorname{Gal}(L/F)$ is prosolvable.
		\end{proof}
		\begin{defn}
			Let $K/F$ be a separable algebraic extension.
			\begin{enumerate}
				\item $K/F$ is called \textit{locally solvable} if every finite subextension $L/F$ of $K/F$ is solvable.
				\item $K/F$ is called \textit{locally solvable by radicals} if every finite subextension $L/F$ of $K/F$ is solvable by radicals.
			\end{enumerate}
		\end{defn}
		\begin{rem}
			Observe that by Galois Theorem for finite extensions, $K/F$ is locally solvable if and only if it is locally solvable by radicals. 
		\end{rem}
		\begin{lem}
			Let $K/F$ be a separable algebraic extension, then $K/F$ is locally solvable if and only if it is prosolvable. Consequently, $K/F$ is locally solvable by radicals if and only if it is solvable by radicals.
			
		\end{lem}
		
		\begin{proof}
			First assume that $K/F$ is prosolvable. Let $N/F$ be the Galois closure of $K$ and $L/F$ be some finite subextension of $K$. Then $L'/F$, the Galois closure of $L/F$, is embedded into $N/F$ and there is an epimorphism $\operatorname{Gal}(N/F)\to \operatorname{Gal}(L'/F)$. Hence, $\operatorname{Gal}(L'/F)$ is solvable. On the other hand, if $N/F$ is the Galois closure of $K/F$, then $N=\bigcup L'/F$ the union of the Galois closures of all finite subextensions $L/F$ inside a given separable closure of $F$. Thus $\operatorname{Gal}(N/F)=\invlim \operatorname{Gal}(L'/F)$ is prosolvable.
		\end{proof}
		\begin{cor}
			Any radical extension of $F$ is locally radical. In other words, if a separable element $x$ over $F$ can be obtained by "$\mu$ steps" of extracting $n$'th roots for some ordinal $\mu$ and natural numbers $n$, then $x$ can be obtained by finite steps of extracting $n$'th roots. 
		\end{cor}
		
		\section{Topological vs. Abstract Composition Factors}
		In this section we study the connection between the composition factors of a profinite group and its profinite completion. Recall that the profinite completion of an abstract group $G$, which we denote by $\hat{G}$, is equal to the inverse limit of all its finite quotients with the natural projections between them. Then $\hat{G}=\varprojlim_{N\unlhd _fG}G/N$. Observe that the finite topological quotients of $\hat{G}$ are precisely the finite abstract quotients of $G$. For more information about profinite completions see \cite[Section 3.2]{ribes2000profinite}.
		
		This definition can in particular be applied for profinite groups $G$. Clearly, $G=\hat{G}$ if and only if all finite-index (normal) subgroups of $G$ are open in $G$. Such groups are called \textit{strongly complete} groups. In that case, properties of the complete system of finite abstract quotients of $G$ are reflected in properties of the
		topological group $G$.
		
		Although it seems a-priory unlikely to hold in general for all profinite groups- it turns out that some topological properties of a general profinite group are in fact properties of the complete system of its finite abstract quotients, or equivalently, are preserved under profinite completion.
		
		Some basic examples are being pro-$\pi$ (prosolvable). By \cite[Proposition 4.2.3 and Corollary 4.2.4]{ribes2000profinite}, every abstract finite quotient of a pro-$\pi$ (prosolvable) group is a finite $\pi$-group (solvable group). Here $\pi$ denotes a set of primes and by (pro-)$\pi$ groups we refer to (inverse limits of) finite groups whose orders are only divided by primes for $\pi$. In particular, the profinite completion of a pro-$\pi$ (prosolvable) group is again a pro-$\pi$ (prosolvable) group. Other interesting connections between a profinite group and its profinite completion can be found, for example, in \cite{bar2024profinite}, \cite{bar2022profinite} and \cite{bar2024cohomological}. The importance of understanding those connections between $G$ and $\hat{G}$ lies in the fact that the profinite completion of a profinite group encodes, in some sense, the abstract group-theoretic properties of the group which are not immediate topological properties. Understanding these abstract group-theoretic properties of a profinite group helps us moving toward solution of the following question that was posed in \cite{bar2024groups}:
		\begin{question}
			Which abstract groups can be realized as Galois groups?
		\end{question}
		By the famous correspondence due to Krull and Leptin, this question is equivalent to: what abstract group can be given a profinite group-topology?
		
		We say that a finite group is \emph{anabelian} if all its composition factors are non-abelian. Let us say that a profinite group is anabelian if all its composition factors are non-abelian. By Proposition \ref{well defined composition factors} this is equivalent to being an inverse limit of anabelian finite groups. Inspired by the prosolvable case we suggest the following question:
		
		\begin{question}\label{question 1}
			Let $G$ be an anabelian profinite group. Is $\hat{G}$ an anabelian profinite group as well?
		\end{question}
		
		Some evidence for a positive answer might be found in the following proposition:
		\begin{prop}
			Let $G$ be an anabelian profinite group. Then $G$ has no finite abstract abelian quotients.
		\end{prop}
		\begin{proof}
			Since $G$ is anabelian it has no topological finite abelian quotients. Hence there is no proper open subgroup ${[G,G]}\leq U\unlhd_o G$. In particular there is no proper open subgroup $\overline{[G,G]}\leq U\unlhd_o G$. Recall that every closed normal subgroup of $G$ is the intersection of all open normal subgroups containing it, see \cite[Proposition 2.1.4 (d)]{ribes2000profinite}. We get that $\overline{[G,G]}=G$.
			
			We will show a stronger result: in fact, $G$ has no abelian quotients at all. Equivalently: $G=[G,G]$. By the preceding discussion it is enough to show that $[G,G]=\overline{[G,G]}$.
			For that we briefly recall some observations that were made in  \cite[Pg. 175]{nikolov2007finitely}.
			
			\sloppy
			Let $w=w(x_1,...,x_k)$ be a group word. Denote by $G^{\{w\}}=\{w(g_1,...,g_k)^{\pm 1}|g_1,...,g_k\in G\}$, $(G^{\{w\}})^{*n}=\{s_1\cdots s_n|s_1,...,s_n\in G^{\{w\}}\}$ and $w(G)=\bigcup_{n\in \mathbb{N}}  (G^{\{w\}})^{*n}$ the abstract subgroup that is generated by all substitutions of elements from $G$ in $w(x_1,...,x_k)$. Then $w(G)$ is closed if and only if there exists some $n\in \mathbb{N}$ such that $w(G)= (G^{\{w\}})^{*n}$. Moreover, $w(G)= (G^{\{w\}})^{*n}$ if and only if $w(G/N)=(G/N^{\{w\}})^{*n}$ for every open normal subgroup $N\unlhd_oG$. 
			
			By \cite[Theorem 2]{nikolov2016verbal} there exists some constant $D$ such that every element in an anabelian finite group is a product of $D$ commutators. By definition of an anabelian profinite group, we are done.  
		\end{proof}
		It is worth mentioning that a profinite group might have finite abstract quotients that does not appear as topological quotients, as described in the following example from \cite[Pg. 176]{nikolov2007finitely}.
		\begin{exam}
			In \cite[Lemma 2.2]{holt1989enumerating}  the author constructed for every prime $p$ and $q$ a power of $p$ a sequence of finite perfect groups $K_n$ such that for each $n$, $K_n\cong P_n\rtimes SL_2(q)$ for $P$ a finite $p$-group, and $\log|K_n|/\log|K_n^{\{w\}}\to \infty$ for $w=x^2$. Let $G=\prod_n K_n$. Then by \cite[Pg. 176]{nikolov2007finitely} $G=\overline{[G,G]}$ and thus has no topological finite abelian quotients. However, since $G^2\lneq G$, $C_2$ appears as an abstract finite quotient of $G$. One observes that $C_2$ does appear as a (topological) composition factor of $G$, since $C_2\cong Z(SL_2(q))$.
		\end{exam}
		This raises a question regarding the possible connection between the \textit{abstract} and \textit{topological} composition factors of a profinite group $G$. Let $S$ be finite simple group. We say that $S$ is an \textit{abstract} (respectively \textit{topological}) composition factor of $G$ if it appears as a composition factor of some abstract (respectively topological) finite quotient of $G$.
		\begin{question}\label{topological vs abstract composition factors}
			Does every abstract composition factor of a profinite group $G$ also occur as a topological composition factor of $G$?
		\end{question}
		Since the abstract finite quotients of a profinite group are precisely the topological finite quotients of its profinite completion, by Proposition \ref{well defined composition factors} Question \ref{topological vs abstract composition factors} is equivalent to the following:
		\begin{myquestion}{\ref{topological vs abstract composition factors}*}
			Let $G$ be a profinite group. Do $G$ and $\hat{G}$ have the same composition factors?
		\end{myquestion}
		Although Question \ref{topological vs abstract composition factors} remains widely open, we can prove some partial results. The first one says that a group with large topological composition factors cannot have an abstract finite image which is too small.
		
		\begin{prop} \label{large} There is a function $f: \mathbb N \rightarrow \mathbb N$ with the following property. Suppose that $L$ is an abstract quotient of a profinite group $G$. Then $G$ has a topological composition factor $L_0$ with $|L_0|< f(|L|)$. 
		\end{prop}
		
		\begin{proof} 
			Let $q=|L|$. By Proposition 10.1 of \cite{nikolov2007finitely} there is are integers $C=C(q)$ and $m=m(q)$ with the following property: 
			
			Let $\Gamma$ be a finite group with a semisimple normal subgroup $N$ which is a product of simple groups, each of size at least $C$. For any $g_1, \ldots, g_m$ in $G$ we have 
			\begin{equation}\label{powers} \{ (g_1z_1)^q \cdots (g_mz_m)^q  \ | \ z_i \in N , i=1, \ldots, m \} =g_1^qg_2^q \cdots g_m^q N. \end{equation}
			
			Set $f=f(|L|):=\max \{ C(q),q+1\}$ and suppose that every topological composition factor of $G$ has size at least $f$. We claim that \begin{equation} \label{claim} G=\left \{ g_0 g_1^q g_2^q \cdots g_m^q \ | \ g_i \in G \textrm{ and } o(g_0) \in \hat{\mathbb Z} \textrm{ is coprime to } q \right \}.\end{equation} 
			
			Assuming this claim it follows that $G=G^q$ and since $L^q=\{1\}$ we conclude that $L$ cannot be an abstract quotient of $G$, contradiction. Hence $G$ must have a topological composition factor of size at least $f$.
			
			It remains to prove the (\ref{claim}). The right hand side is a closed subset of $G$ and thus it is enough to show that it is dense in $G$. Therefore we may assume from now on that $G$ is a finite group. We argue by induction on $|G|$ starting with the case when $|G|=1$ when the equality (\ref{claim}) is trivially true. If $G$ is simple we set $N=G$, otherwise we choose $N$ to be a minimal normal subgroup of $G$. Let $x \in G$. By the induction hypothesis there are elements $g_0, \ldots, g_m \in G$ with order of $g_0N$ in $G/N$ coprime to $q$ and such that $x \equiv g_0g_1^q \cdots g_m^q$ mod $N$. If $N$ is semisimple then its simple factors are of size at least $C(q)$ by assumption. In that case the existence of $z_1, \ldots, z_m \in N$ with $x=g_0 (g_1z_1)^q \cdots (g_mz_m)^q$ follows directly from (\ref{powers}). If $N$ is elementary abelian then it must be a $p$-group for some prime $p>q$ and in particular $|N|$ is coprime to $q$. In that case any element of $g_0N \subset G $ has order in $G$ coprime to $q$ and we can write $x=g_0'g_1^q \cdots, g_m^q$ where $g_0':=x(g_1^q \cdots, g_m^q)^{-1} \in g_0N$ as required. The claim (\ref{claim}) is proved.
		\end{proof}
		
		The next Proposition shows that if $L$ is an abstract composition factor of $G$ then the topological composition factors of $G$ cannot be all smaller than $L$. Recall that for groups $A,B$ we say that $A$ is a section of $B$ if $A \cong C/D$ where $D \leq C \leq B$ with $D$ normal in $C$. It is easy to show that if $H$ is a closed subgroup of a profinite group $G$ then every topological composition factor of $H$ occurs as a section of some topological composition factor of $G$.
		
		\begin{prop} \label{small} Let $L$ be an abstract composition factor of a profinite group $G$. Then $G$ has a topological composition factor $L_0$ such that $L$ is a section of $L_0$.
		\end{prop}
		\begin{proof}
			Let $N$ be a normal subgroup of finite index in $G$ such that  $L$ is a composition factor of $G/N$. Choose $a_1, \ldots, a_d \in G$ such that $a_1N, \ldots, a_dN$ generate $G/N$ and let $H:=\overline{\langle a_1, \ldots, a_d \rangle} \leq G$. Note that $HN=G$ and hence $G/N \simeq H/(N \cap H)$. On the other hand $H$ is a finitely generated profinite group and hence strongly complete by \cite{nikolov2007finitely}. Therefore $L$ occurs as a topological composition factor of $H$ and hence occurs a section in some topological composition factor of $G$.   
		\end{proof}
		
		In particular since all proper subgroups of the alternating group $A_5$ are solvable we can deduce the following.
		\begin{cor} Let $G$ be a profinite group whose nonabelian topological composition factors are equal to $A_5$. Then the nonabelian abstract composition factors of $G$ must also be equal to $A_5$.
			
		\end{cor}

		\bibliographystyle{plain}

	\end{document}